\documentclass[12pt,reqno]{amsart}
\usepackage{amsmath}

\usepackage{graphicx}
\newtheorem{defi}{Definition}
\newcommand{\brdef}{\begin{defi}}
\newcommand{\erdef}{\end{defi}}

\newtheorem{cor}{Corollary}
\newcommand{\bcor}{\begin{cor}}
\newcommand{\ecor}{\end{cor}}

\newtheorem{thm}{Theorem}
\newcommand{\bth}{\begin{thm}}
\newcommand{\eth}{\end{thm}}
\newtheorem{lem}{Lemma}
\newcommand{\ble}{\begin{lem}}
\newcommand{\ele}{\end{lem}}

\def\pn{\par\noindent}

 \huge

\usepackage{hyperref}
\usepackage{mathrsfs}

\numberwithin{equation}{section}

\begin{document}
\begin{center}
{\large \bf Certain results on $(LCS)_{n}$-manifolds} \\
\
\
\\
{Vishnuvardhana. S. V.\footnote{ corresponding author.}$^a$, Venkatesha$^a$ and S. K. Hui$^b$}\\
$^a$ Department of Mathematics, Kuvempu University,\\
Shankaraghatta - 577 451, Shimoga, Karnataka, INDIA.\\
\pn{\tt e-mail:{\verb+svvishnuvardhana@gmail.com+, \verb+vensmath@gmail.com+}}\\
$^b$ Department of Mathematics, Bankura University,\\
Bankura - 722 155, West Bengal, INDIA.\\
\pn{\tt e-mail:{\verb+shyamal_hui@yahoo.co.in+}}
\end{center}
\begin{quotation}
{\bf Abstract:} The purpose of the present paper is to study semi-generalized recurrent, semi-generalized Ricci recurrent and conformal Ricci soliton on $(LCS)_{n}$-manifold.
\end{quotation}
\section{Introduction}
\par As a generalization of LP-Sasakian manifold (\cite{KMatsumoto}, \cite{IMihaiRRosca}), Shaikh (\cite{AAShaikh1}, \cite{AAShaikhKKBaishya1}, \cite{AAShaikhKKBaishya2}) introduced the notion of $(LCS)_{n}$-manifolds along with their existence and applications to the general theory of relativity and cosmology. Moreover, Shaikh and his coauthors (\cite{AAShaikh1}-\cite{AASSKHui}) studied $(LCS)_{n}$-manifolds by imposing various curvature restrictions. The $(LCS)_{n}$-manifolds have also been studied by Atceken \cite{MAtceken}, Hui et. al (\cite{ATEHUI}, \cite{CHS}, \cite{SKH}, \cite{SHMA}, \cite{SKH-DC}, \cite{HUC}), Narain and Yadav \cite{DNarainSYadav}, Prakasha \cite{DGPrakasha}, Sreenivasa et al.\cite{GTSVCSB}, Venkatesha and Kumar \cite{VRTNK}, Yadav et al.\cite{SKYPKDDS}, and others.

\par Locally symmetric manifolds were weakened by many geometers in different extents. In those, the idea of recurrent manifolds was introduced by Walker in 1950 \cite{AGWalker}. On the other hand, De and Guha \cite{UCDeNGuha} introduced generalized recurrent manifold ($GK_n$) with the non-zero 1-form $A$ and another non-zero associated 1-form $B$. If the associated 1-form $B$ becomes zero, then the manifold $GK_n$ reduces to a recurrent manifold ($K_n$) introduced by Ruse \cite{HSRuse} and Walker \cite{AGWalker}.

\par The notion of recurrent manifolds have been generalized by various authors as Ricci-recurrent manifolds $(R_{n})$ by Patterson \cite{EMPatterson},  2-recurrent manifolds by Lichnerowicz \cite{ALinchnerowicz}, projective 2-recurrent manifolds by Ghosh \cite{DGhosh} and  generalized Ricci recurrent manifold $(GR_{n})$ by De et al \cite{UCDeNGuhaDKamilya} etc.

\par Recently, semi generalized recurrent condition was introduced and studied on Lorentzian $\alpha$-Sasakian manifolds and P-Sasakian manifolds by Dey and Bhattacharyya \cite{SDAB} and Singh et. al \cite{ASJPSRK} respectively.
\begin{defi}
\par A Riemannian manifold $(M^n, g)$ is said to be semi-generalized recurrent manifold if
\begin{equation}
 (\nabla_{X}R)(Y, Z)W=A(X)R(Y, Z)W+ B(X)g(Z, W)Y, \label{I1}
\end{equation}
holds. Here $A$ and $B$ are the 1-forms defined by
\begin{equation}
A(X)=g(X, \rho_{1}), \,\,\,\,\,\,\,\,B(X)=g(X, \rho_{2}). \label{I2}
\end{equation}
\end{defi}
\begin{defi}
\par A Riemannian manifold $(M^n, g)$ is said to be semi-generalized Ricci recurrent manifold if
\begin{equation}
 (\nabla_{X}S)(Y, Z)=A(X)S(Y, Z)+n B(X)g(Y, Z), \label{I3}
\end{equation}
holds, where $A$ and $B$ are 1-forms and are defined as in (\ref{I2})
\end{defi}
Our work is structured as follows: In section 2, we give a brief information about $(LCS)_{n}$-manifolds. The next three sections are respectively devoted to the study of semi generalized recurrent, semi generalized $\phi$-recurrent and semi generalized Ricci recurrent conditions on $(LCS)_{n}$-manifolds. We have presented an example to verify our result in section~5. Section~6 deals with conformal Ricci solitons on $(LCS)_{n}$-manifold.
Here we mainly studied conformal Ricci soliton in $(LCS)_{n}$-manifold satisfying $R(\xi, X)\cdot \tilde{M}=0$ and $C(\xi, X)\cdot S=0$.
\section{preliminaries}
An $n$-dimensional Lorentzian manifold $M$ is a smooth connected paracompact Hausdorff manifold with a Lorentzian metric $g$, that is,
$M$ admits a smooth symmetric tensor field $g$ of type (0,2) such that for each point $p\in M$, the tensor
$g_{p}:T_{p}M\times T_{p}M \longrightarrow R$ is a non-degenerate inner product of signature $(-,+, \cdots ,+)$, where $T_{p}M$ denotes the tangent space of $M$ at $p$ and $R$ is a real number space. A non-zero vector $v\in T_{p}M$ is said to be timelike (resp., non-spacelike, null, spacelike) if it satisfies $g_{p}(v,v)<0$ $(resp., \leq 0, =0, >0)$ \cite{BONeill}.
\begin{defi}
In a Lorentzian manifold $(M, g)$ a vectorfield $P$ defined by $g(X, P)=A(X)$, for any $X\in \Gamma(TM)$, is said to be a concircular vector field if $$(\nabla_{X}A)(Y)=\alpha\{g(X,Y)+\omega(X)A(Y)\},$$ where $\alpha$ is a non-zero scalar and $\omega$ is a closed 1-form and $\nabla$ denotes the operator of covariant differentiation of $M$ with respect to the Lorentzian metric $g$.
\end{defi}
\indent Let $M$ be an $n$-dimensional Lorentzian manifold admitting a unit timelike concircular vectorfield $\xi$, called the characteristic vector field of the manifold. Then we have
 \begin{equation}
 g(\xi, \xi)=-1, \label{1.1}
 \end{equation}
\par Since $\xi$ is a unit concircular vector field, it follows that there exists a non-zero 1-form $\eta$ such that for
\begin{equation}
 g(X, \xi)=\eta(X), \label{1.2}
 \end{equation}
the equation of the following form holds for all vector fields $X, Y,$
\begin{equation}
 (\nabla_{X}\eta)(Y)=\alpha\{g(X,Y)+\eta(X)\eta(Y)\}, \,\,\,\, \alpha\neq 0 \label{1.3}
 \end{equation}
 where $\alpha$ is a non-zero scalar function satisfying
\begin{equation}
 \nabla_{X}\alpha=X(\alpha)=d\alpha(X)=\rho \eta(X), \label{1.4}
 \end{equation}
$\rho$ being a certain scalar function given by $\rho=-(\xi\alpha)$. Let us take
\begin{equation}
\phi X=X(\alpha)=\frac{1}{\alpha}\nabla_{X}\xi, \label{1.5}
 \end{equation}
 then by virtue of (\ref{1.3}) and (\ref{1.5}), we have
 \begin{equation}
\phi X=X+\eta(X)\xi \label{1.6}
 \end{equation}
from which it follows that $\phi$ is a symmetric (1,1) tensor, called the structure tensor of the manifold.
Thus the Lorentzian manifold $M$ together with the unit timelike concircular vectorfield $\xi$, its associated 1-form $\eta$ and an (1,1) tensorfield $\phi$ is said to be a Lorentzian concircular structure manifold (briefly, $(LCS)_{n}$-manifold) \cite{AAShaikh1}. Especially, if we take $\alpha=1$, then we obtain the LP-Sasakian structure of Matsumoto \cite{KMatsumoto}.\\
The following relations holds in a $(LCS)_{n}$-manifold $(n>2)$ (\cite{AAShaikh1}, \cite{AAShaikhKKBaishya1}):
\begin{eqnarray}
& &\label{1.7} \phi^2=I+\eta  \circ  \xi,\\
& &\label{1.8}\eta(\xi)=-1, \,\,\,\,\, \phi \xi=0, \,\,\,\,\,\, \eta \circ \phi =0, \,\,\,\,\,\,g(X, \xi)=\eta(X),\\
& &\label{1.9}  g(\phi X,\phi Y)=g(X,Y)+ \eta(X)\eta(Y),\\
& &\label{1.10}  (\nabla_{X}\phi)Y=\alpha \{g(X, Y)\xi +2\eta(X)\eta(Y)\xi+ \eta(Y)X\},\\
& &\label{1.11}  \eta(R(X, Y)Z)=(\alpha^2-\rho)\{g(Y, Z)\eta(X) - g(X, Z)\eta(Y)\},\\
& &\label{1.12}  R(X, Y )\xi = (\alpha^2-\rho) \{\eta(Y) X-\eta(X) Y\},\\
& &\label{1.13}  R(\xi, X )Y = (\alpha^2-\rho) \{g(X, Y)\xi-\eta(Y) X\},\\
& &\label{1.14}  S(X, \xi)=(n-1)(\alpha^2-\rho)\eta(X).
\end{eqnarray}
\section{Semi-generalized recurrent $(LCS)_{n}$-manifold}
\begin{defi}
\par An $(LCS)_{n}$-manifold $(M^n, g)$ is said to be semi-generalized recurrent manifold if
\begin{equation}
\label{sgr2} (\nabla_{X}R)(Y, Z)W=A(X)R(Y, Z)W+ B(X)g(Z, W)Y,
\end{equation}
holds.
\end{defi}
\begin{thm}
In a semi-generalized recurrent $(LCS)_{n}$-manifold the scalar curvature $r$ is given by
\begin{equation}
 r = \frac{1}{A(\xi)}\{ 2(n-1)(\alpha^{2}-\rho) \eta(\rho_{1})-(n^{2}+2) B(\xi) \}. \label{sgr3}
\end{equation}
\end{thm}
\begin{proof}
Let us consider a semi-generalized recurrent $(LCS)_{n}$-manifold $(M^n, g)$.
\par Taking cyclic sum of (\ref{sgr2}) with respect to $X, Y, Z$ and making use of first Bianchi's identity, we have
\begin{eqnarray}
 \nonumber A(X)R(Y, Z)W + B(X)g(Z, W)Y + A(Y)R(Z, X)W \\
\label{sgr4}    + B(Y)g(X, W)Z + A(Z)R(X, Y)W + B(Z)g(Y, W)X=0.
\end{eqnarray}
On contraction, (\ref{sgr4}) leads to
\begin{eqnarray}
\nonumber A(X)S(Z, W) + n B(X)g(Z, W) + A(R(Z, X)W) \\
\label{sgr5}    + B(Z)g(X, W) - A(Z)S(X, W) + B(Z)g(X, W)=0.
\end{eqnarray}
Again contracting the above equation, we get
\begin{equation}
 r A(X) + (n^{2}+2) B(X) - 2S(X, \rho_{1}) =0. \label{sgr6}
\end{equation}
On substituting $X=\xi$ in (\ref{sgr6}), one can obtain the desired result.
\end{proof}
Now from (\ref{sgr2}), we have
\begin{equation}
\label{sgr7} (\nabla_{X}R)(Y, Z, W, U) = A(X)g(R(Y, Z)W, U) + B(X)g(Z, W)g(Y, U).
\end{equation}
On contracting (\ref{sgr7}) twice and taking non zero constant scalar curvature, we get
\begin{equation}
\label{sgr8} A(X)=-\frac{n^{2}}{r} B(X).
\end{equation}
Thus, we have the following assertion;
\begin{thm}
In a semi-generalized recurrent $(LCS)_{n}$-manifold $(M^{n}, g)$  the associated 1-forms are in opposite direction provided scalar curvature is positive constant.
\end{thm}
\section{Semi-generalized $\phi$-recurrent $(LCS)_{n}$-manifold}
\begin{defi}
\par An $(LCS)_{n}$-manifold $(M^n, g)$ satisfying the condition
\begin{equation}
\label{sgpr1} \phi^{2}((\nabla_{W}R)(X, Y)Z)=A(W)R(X, Y)Z + B(W)g(Y, Z)X,
\end{equation}
is called a semi-generalized $\phi$-recurrent manifold.
\end{defi}
Using (\ref{1.7}) and then taking inner product with $U$ in (\ref{sgpr1}), we obtain
\begin{eqnarray}
 \nonumber g((\nabla_{W}R)(X, Y)Z, U) + \eta((\nabla_{W}R)(X, Y)Z)\eta(U)=A(W) g(R(X, Y)Z, U)\\
\label{sgpr2}+ B(W)g(Y, Z)g(X, U).
\end{eqnarray}
Contraction of the above equation over $X$ and $U$, reduces to
\begin{equation}
\label{sgpr3} (\nabla_{W}S)(Y, Z) + \eta((\nabla_{W}R)(\xi, Y)Z)=A(W) S(Y, Z)+ n B(W)g(Y, Z).
\end{equation}
In an $n$-dimensional $(LCS)_{n}$-manifold, we can easily see that
\begin{equation}
\label{sgpr4} g((\nabla_{W}R)(\xi, Y)Z, \xi) = -(2\alpha \rho-\beta)\{g(Y, Z)+ \eta(Y) \eta(Z) \}\eta(W).
\end{equation}
By virtue of (\ref{sgpr4}), (\ref{sgpr3}) yields
\begin{equation}
\label{sgpr5} (\nabla_{W}S)(Y, Z) = (2\alpha \rho-\beta)\{g(Y, Z)+ \eta(Y) \eta(Z) \}\eta(W) + A(W) S(Y, Z)+ n B(W)g(Y, Z).
\end{equation}
Taking an account of (\ref{1.5}), the covariant derivative of (\ref{1.14}) can be written as
\begin{equation}
\label{sgrr4} (\nabla_{X}S)(Y, \xi)=(n-1)[ (2\alpha \rho-\beta)\eta(X)\eta(Y)+(\alpha^{2}-\rho)\alpha g(X, Y)]-\alpha S(X, Y).
\end{equation}
Setting $Z=\xi$ in (\ref{sgpr5}) and making use of (\ref{sgrr4}), one can arrive at
\begin{eqnarray}
 \nonumber S(Y, W) &=& - \frac{\eta(Y)}{\alpha}[(n-1)(\alpha^{2}-\rho) A(W) + n B(W)]\\
\label{sgpr6}  &+& \frac{(n-1)}{\alpha}[(2 \alpha \rho-\beta)\eta(Y)\eta(W) + (\alpha^{2}- \rho)\alpha g(Y, W)].
\end{eqnarray}
Replacing $Y$ by $\phi Y$ in (\ref{sgpr6}) gives
\begin{equation}
\label{sgpr7} S(Y, W) = (n-1)(\alpha^{2}- \rho) g(Y, W).
\end{equation}
Hence, we can state the following;
\begin{thm}
A semi-generalized $\phi$-recurrent $(LCS)_{n}$-manifold $(M^{n}, g)$ is an Einstein manifold.
\end{thm}
\section{Semi-generalized Ricci recurrent $(LCS)_{n}$-manifold}
\begin{defi}
\par A $(LCS)_{n}$-manifold $(M^n, g)$ is said to be semi-generalized Ricci recurrent manifold if
\begin{equation}
\label{sgrr2} (\nabla_{X}S)(Y, Z)=A(X)S(Y, Z)+n B(X)g(Y, Z),
\end{equation}
holds.
\end{defi}
Suppose that $(LCS)_{n}$-manifold is semi-generalized Ricci recurrent. Then from (\ref{sgrr2}) we have
\begin{equation}
\label{sgrr3} (\nabla_{X}S)(Y, \xi)=A(X)S(Y, \xi)+n B(X)g(Y, \xi).
\end{equation}
By virtue of (\ref{sgrr4}) and (\ref{sgrr3}), we have
\begin{equation}
\label{sgrr5} \alpha S(X, Y)=(n-1)[ (2\alpha \rho-\beta)\eta(X)\eta(Y)+(\alpha^{2}-\rho)\alpha g(X, Y) - (\alpha^{2}-\rho)A(X)\eta(Y)]-n B(X)\eta(Y).
\end{equation}
Replacing $Y$ by $\phi Y$ in (\ref{sgrr5}) and using (\ref{1.14}), we have
\begin{equation}
\label{sgrr6}  S(X, Y)=(n-1)(\alpha^{2}-\rho) g(X, Y).
\end{equation}
If the manifold under consideration is Einstein, then (\ref{sgrr6}) implies $\alpha^{2}-\rho=$
constant and hence $2\alpha \rho- \beta= 0$: Conversely, if $2\alpha \rho- \beta= 0$, then $\nabla_{X}(\alpha^{2}-\rho)=0$: Consequently $\alpha^{2}-\rho=$ constant.
\par Thus we state the following:
\begin{thm}
A semi-generalized Ricci recurrent $(LCS)_{n}$-manifold is Einstein if and only if $\beta=2\alpha \rho$.
\end{thm}
\noindent{\bf Example 5.1} We consider a 3-dimensional manifold $M = \{(x, y, z)\in R^3\}$, where $(x, y, z)$
are the standard coordinate in $R^3$. Let $\{E_{1}, E_{2}, E_{3}\}$ be linearly independent global frame field on $M$ given by
\begin{equation}
\nonumber E_{1}= z(x\frac{\partial}{\partial{x}}+y\frac{\partial}{\partial{y}}), \,\,\,\,\,\,\,\,\,\,\, E_{2}= z\frac{\partial}{\partial{y}}, \,\,\,\,\,\,\,\,\,\, E_{3}= \frac{\partial}{\partial{z}}.
\end{equation}
\par Let $g$ be the Lorentzian metric defined by
\begin{eqnarray}
\nonumber g(E_{1}, E_{2})=g(E_{1}, E_{3})=g(E_{2}, E_{3})= 0,\\
\nonumber g(E_{1}, E_{1})=g(E_{2}, E_{2})=1, g(E_{3}, E_{3})= -1.
\end{eqnarray}
Let $\eta$ be the 1-form defined by $\eta(Z)=g(Z, e_{3})$ for any vector field $Z\in \chi(M)$.
\par Let $\phi$ be the (1, 1)-tensor field defined by
\begin{equation}
\nonumber \phi (E_{1})= E_{1}, \,\,\,\,\,\, \phi (E_{2})= E_{2}, \,\,\,\,\, \phi (E_{3})= 0.
\end{equation}
Using the linearity of $\phi$ and $g$, we have
\begin{eqnarray}
\nonumber & & \eta(E_{3})=-1, \\
\nonumber & & \phi^{2}U=U + \eta(U)E_{3}, \\
\nonumber & & g(\phi U, \phi V)=g(U, V)+\eta(U) \eta(V),
\end{eqnarray}
for any $U, V \in \chi(M)$. Thus for $E_{3}=\xi$, $(\phi, \xi, \eta, g)$ defines a Lorentzian paracontact
structure on $M$.\\
\par Now we have
\begin{equation}
\nonumber [E_{1}, E_{2}]=-zE_{2}, \,\,\,\,\,\, [E_{1}, E_{3}]=-\frac{1}{z} E_{1}, \,\,\,\,\, [E_{2}, E_{3}]=-\frac{1}{z} E_{2}.
\end{equation}
Let $\nabla$ be the Levi-Civita connection with respect to the Lorentzian metric $g$. Using Koszul formula for the Lorentzian metric $g$, we can easily
calculate
\begin{eqnarray}
\nonumber & & \nabla_{E_{1}}E_{1}=-\frac{1}{z} E_{3} , \,\,\,\,\,\,\,\,\, \nabla_{E_{2}}E_{1}=zE_{2} , \,\,\,\,\,\,\,\,\,\,\,\,\,\,\,\,\, \nabla_{E_{3}}E_{1}=0,  \\
\nonumber & &  \nabla_{E_{1}}E_{2}= 0, \,\,\,\,\,\,\,\,\,\,\,\,\,\,\,\,\,\,\, \nabla_{E_{2}}E_{2}= -\frac{1}{z}E_{3}-zE_{1}, \,\,\,\,\,\,\,\,\, \nabla_{E_{3}}E_{2}=0, \\
\nonumber & & \nabla_{E_{1}}E_{3}= -\frac{1}{z}E_{1}, \,\,\,\,\,\,\,\,\,\,\,\,\,\,\, \nabla_{E_{2}}E_{3}= -\frac{1}{z}E_{2} , \,\,\,\,\,\,\,\,\,\,\,\,\,\, \nabla_{E_{3}}E_{3}=0 .
\end{eqnarray}
From the above, it can be easily seen that $(\phi, \xi, \eta, g)$ is an $(LCS)_{3}$ structure on $M$.
Consequently $M^{3}(\phi, \xi, \eta, g)$ is an $(LCS)_3$-manifold with $\alpha =-\frac{1}{z}\neq 0$ and $\rho =-\frac{1}{z^2}$.

From the above relations, one can have the non-vanishing components of the curvature tensor as
\begin{eqnarray}
\nonumber & & R(E_{2}, E_{3})E_{3}=-\frac{2}{z^2}E_{2}, \,\,\,\,\,\,\,\,\,\,\,\,\,\,\,\, R(E_{1}, E_{3})E_{3}=-\frac{2}{z^2}E_{1}, \,\,\,\,\,\,\,\,\,\,\,\,\,\,\,\,\,\, R(E_{1}, E_{2})E_{2}=\frac{1}{z^2}E_{1}-z^{2}E_{1}, \\
\nonumber & & R(E_{2}, E_{3})E_{2}=-\frac{2}{z^2}E_{3}, \,\,\,\,\,\,\,\,\,\,\,\,\,\, R(E_{1}, E_{2})E_{1} =z^{2}E_{2}-\frac{1}{z^2}E_{2}, \,\,\,\,\, R(E_{1}, E_{3})E_{1} =-\frac{2}{z^2}E_{3}.
\end{eqnarray}
Using this, we find the values of the Ricci tensor as follows:
\begin{equation}
\nonumber S(E_{1}, E_{1}) = -(z^2+\frac{1}{z^2}), \,\,\,\,\,\,\,\,\,\, S(E_{2}, E_{2}) =-(z^2+\frac{1}{z^2})  , \,\,\,\,\,\,\,\,\,\, S(E_{3}, E_{3}) =-\frac{4}{z^2}.
\end{equation}
Since ${E_{1}, E_{2}, E_{3}}$ forms a basis, any vector field $X, Y, Z \in \chi(M)$ can be
written as \\ $X = a_{1}E_{1} + b_{1}E_{2} + c_{1}E_{3}$, $Y = a_{2}E_{1} + b_{2}E_{2} + c_{2}E_{3}$, where $a_{i}, b_{i}, c_{i} \in R^{+}$ (the set of all positive real numbers), $i = 1, 2$. This
implies that
\begin{equation}
\nonumber S(X, Y) = -(a_{1}a_{2}+b_{1}b_{2})(z^2+\frac{1}{z^2})-c_{1}c_{2}\frac{4}{z^2} \,\,\,\,\,\,\,\,\,\,\,\, and\,\,\,\,\,\,\,\,\,\,\, g(X, Y)=a_{1}a_{2} + b_{1}b_{2} - c_{1}c_{2}.
\end{equation}
By virtue of above, we have the following:
\begin{eqnarray}
\nonumber & & (\nabla_{E_{1}}S)(X, Y)=-(a_{1}c_{2}+c_{1}a_{2})(z+\frac{5}{z^3}),\\
\nonumber & & (\nabla_{E_{2}}S)(X, Y)=-(b_{1}c_{2}+c_{1}b_{2})(z+\frac{5}{z^3}),\\
\nonumber & & (\nabla_{E_{3}}S)(X, Y)=0.
\end{eqnarray}
This means that manifold under the consideration is not Ricci symmetric. Let us now consider the 1-forms
\begin{eqnarray}
\nonumber & & A(E_{1})=\frac{(a_{1}c_{2}+c_{1}a_{2})}{z(a_{1}a_{2}+b_{1}b_{2})}, \,\,\,\,\,\,\,\,\, B(E_{1})=\frac{-4(a_{1}c_{2}+c_{1}a_{2})}{3z^{3}(a_{1}a_{2}+b_{1}b_{2})},\\
\nonumber & & A(E_{2})= \frac{(b_{1}c_{2}+c_{1}b_{2})}{z(a_{1}a_{2}+b_{1}b_{2})}, \,\,\,\,\,\,\,\, B(E_{2})= \frac{-4(b_{1}c_{2}+c_{1}b_{2})}{3z^{3}(a_{1}a_{2}+b_{1}b_{2})}, \\
\nonumber & & A(E_{3})= 0, \,\,\,\,\,\,\,\, B(E_{3})= 0,
\end{eqnarray}
at any point $X\in M$. From (\ref{sgrr2}) we have
\begin{equation}
(\nabla_{E_{i}}S)(Y, Z) = A(E_{i})S(Y, Z) + 3B(E_{i}) g(Y, Z), \,\,\,\,\,\,\,\,\,\,\,\, i=1, 2, 3.\label{6.1}
\end{equation}
It can be easily shown that the manifold with the above 1-forms satisfies the relation (\ref{6.1}). Hence the manifold under consideration is a semi generalized Ricci recurrent $(LCS)_{3}$-manifold.
\section{Conformal Ricci Solitons on $(LCS)_{n}$-manifold}
\par In 2004, Fischer \cite{AEFischer} developed conformal Ricci flow.  In a classical Ricci flow equation, unit volume constraint plays an important role but in conformal Ricci flow equation scalar curvature is considered. This Ricci flow equations are called conformal Ricci flow equations because of the role that conformal geometry plays in constraining the scalar curvature and these equations are the vector field sums of a
conformal flow equation and a Ricci flow equation. These new equations are given by
\begin{equation}
\begin{split}
\label{6.2}
\frac{\partial g}{\partial t}+2(S+\frac{g}{n})=-pg,\\
R(g)=-1,
\end{split}
\end{equation}
where $R(g)$ is the scalar curvature of the manifold and $p\geq0$ is a non-dynamical (time-dependent parameter) scalar field, which is also called as conformal pressure and $n$ is the dimension of manifold. The conformal Ricci flow equations are analogous to the Navier-Stokes equations of fluid mechanics (i.e., for incompressible fluid flow) \cite{GCPMDDMMR}. At equilibrium conformal pressure is zero and strictly positive otherwise. And it serves as a Lagrange multiplier to conformally deform the metric flow so as to maintain the scalar curvature constraint.

Basu and Bhattacharyya (\cite{NBAB1}, \cite{NBAB2}), introduced the conformal Ricci soliton equation as
\begin{equation}
\mathscr{L}_{V}g+ 2S_{g} = [2\lambda-(p+\frac{2}{n})] g. \label{6.3}
\end{equation}
This equation is the generalization of the Ricci soliton equation and it also satisfies the
conformal Ricci flow equation.\\
Setting $V=\xi$ in (\ref{6.3}) and using (\ref{1.5}), (\ref{1.6}) and (\ref{1.8}), one can get
\begin{equation}
S(X, Y)=kg(X, Y)-\alpha \eta(X)\eta(Y), \label{6.4}
\end{equation}
where $k=\lambda-(\frac{p}{2}+\frac{1}{n})-\alpha$.

\par Contraction of the above equation gives the scalar $\lambda$ under the conformal Ricci soliton as
\begin{equation}
\lambda=\frac{p}{2}+ \frac{(n+1)}{n}\alpha. \label{6.5}
\end{equation}
Hence we have;
\begin{thm}
A conformal Ricci soliton $(LCS)_{n}$-manifold is an $\eta$-Einstein manifold and the scalar $\lambda$ is given by (\ref{6.5}).
\end{thm}
Now we consider conformal Ricci soliton in $(LCS)_{n}$-manifold satisfying $R(\xi, X)\cdot \tilde{M}=0$.\\
\par Pokhariyal and Mishra \cite{GPPRSM} defined a tensor field $\tilde{M}$ on a Riemannian manifold as
\begin{equation}
\tilde{M}(X, Y )Z = R(X, Y )Z -\frac{1}{2(n-1)}[ S(Y, Z)X -S(X, Z)Y + g(Y, Z)QX - g(X, Z)QY ].
\end{equation}
Such a tensor field $\tilde{M}$ is known as $M$-projective curvature tensor and it bridges the gap between conformal curvature tensor, conharmonic curvature tensor and concircular curvature tensor on one side and $H$-projective curvature tensor on the other.

\par Suppose conformal Ricci soliton in $(LCS)_{n}$-manifold satisfies $R(\xi, X)\cdot \tilde{M}=0$. Which implies that
\begin{eqnarray}
\nonumber \eta(R(\xi, X)\tilde{M}(U, V)W) - \eta(\tilde{M}(R(\xi, X)U, V)W) \\
- \eta(\tilde{M}(U, R(\xi, X)V)W) - \eta(\tilde{M}(U, V)R(\xi, X)W)=0. \label{7.1}
\end{eqnarray}
By virtue of (\ref{1.13}), (\ref{7.1}) turns in to
\begin{eqnarray}
\nonumber (\alpha^2-\rho)[g(X, \tilde{M}(U, V)W) + \eta(\tilde{M}(U, V)W)\eta(X) + g(X, U)\eta(\tilde{M}(\xi, V)W)\\
\nonumber - \eta(U)\eta(\tilde{M}(X, V)W)  + g(X, V)\eta(\tilde{M}(U, \xi)W)- \eta(V)\eta(\tilde{M}(U, X)W)\\
 + g(X, W)\eta(\tilde{M}(U, V)\xi)- \eta(W)\eta(\tilde{M}(U, V)X)]=0. \label{7.2}
\end{eqnarray}
In view of (\ref{1.9}), (\ref{1.11}), (\ref{1.12}), (\ref{1.13}) and (\ref{6.4}), we have
\begin{eqnarray}
\label{7.3} \eta(\tilde{M}(X, Y)Z)&=&\{(\alpha^2-\rho)-\frac{k}{(n-1)}+\frac{\alpha}{2(n-1)}\}\{g(Y, Z)\eta(X)-g(X, Z)\eta(Y)\},\\
\label{7.4} \eta(\tilde{M}(\xi, Y)Z)&=&-\{(\alpha^2-\rho)-\frac{k}{(n-1)}+\frac{\alpha}{2(n-1)}\}\{g(Y, Z)+\eta(Y)\eta(Z)\},\\
\label{7.5} \eta(\tilde{M}(X, Y)\xi)&=&0.
\end{eqnarray}
Suppose that $(\alpha^2-\rho)\neq 0$. Then the equation (\ref{7.2}) with the use of (\ref{7.3})-(\ref{7.5}) gives
\begin{eqnarray}
 \nonumber & & g(X, \tilde{M}(U, V)W) \\
 &+&\{(\alpha^2-\rho)-\frac{k}{(n-1)}+\frac{\alpha}{2(n-1)}\}\{ g(X, V)g(U, W) - g(X, U)g(V, W)\}=0.\label{7.7}
\end{eqnarray}
On contraction of the above equation over $X$ and $U$, one can get
\begin{equation}
S(V, W)=\frac{1}{n}[R(g)-2(\alpha^2-\rho)(n-1)^2+2k(n-1)-\alpha(n-1)]g(V, W).\label{7.8}
\end{equation}
Thus, we arrive at the following assertion:
\begin{thm}
A conformal Ricci soliton in $(LCS)_{n}$-manifold satisfying $R(\xi, X)\cdot \tilde{M}=0$ is an Einstein manifold provided $(\alpha^2-\rho)\neq 0$.
\end{thm}
We now consider conformal Ricci soliton in $(LCS)_{n}$-manifold satisfying $C(\xi, X)\cdot S=0$.\\
\par An interesting invariant of a concircular transformation is the concircular curvature tensor $C$ \cite{KYano}. And it is defined by
\begin{equation}
C(X, Y)W = R(X, Y)W -\frac{r}{n(n - 1)}\{g(Y, W)X - g(X, W)Y\},
\end{equation}
Riemannian manifolds with vanishing concircular curvature tensor are of constant curvature. Thus, the concircular
curvature tensor is a measure of the failure of a Riemannian manifold to be of constant curvature.

Assume that, conformal Ricci soliton in $(LCS)_{n}$-manifold satisfies $C(\xi, X)\cdot S=0$. Then we have
\begin{equation}
 S(C(\xi, X)Y, Z) + S(Y, C(\xi, X)Z)=0.  \label{8.1}
\end{equation}
Using (\ref{1.10}) and (\ref{1.13}) in (\ref{8.1}), we get
\begin{eqnarray}
\nonumber \{n(n-1)(\alpha^2-\rho)+1\}\{g(X, Y)S(\xi, Z) - \eta(Y)S(X, Z) + g(X, Z)S(Y,\xi)\\
 - \eta(Z)S(Y, X)\}=0 \label{8.2}
\end{eqnarray}
Now we proceed the calculations for $n(n-1)(\alpha^2-\rho)+1 \neq 0$. \\
Plugging $Y=\xi$ in (\ref{8.2}) and then using (\ref{6.4}), gives rise to
\begin{equation}
 S(X, Z)= (k-\alpha)g(X, Z) \label{8.3}
\end{equation}
Therefore one can state the following:
\begin{thm}
A conformal Ricci soliton in $(LCS)_{n}$-manifold satisfying $C(\xi, X)\cdot S=0$ is an Einstein manifold provided $n(n-1)(\alpha^2-\rho)+1\neq 0$.
\end{thm}
{\bf{Acknowledgement:}} Vishnuvardhana S.V. was supported by the Department of Science and Technology, India through
the SRF [IF140186] DST/INSPIRE FELLOWSHIP/2014/181.

\end{document}